%% file: main.tex
\documentclass[12pt, english]{article}
\usepackage[left=3cm,right=2.5cm,top=3cm,bottom=3cm]{geometry}
\usepackage{mathtools}
\usepackage{amsfonts}
\usepackage{amssymb}
\usepackage{wasysym}
\usepackage{color}
\usepackage[english]{babel}
\usepackage[utf8]{inputenc}
\usepackage{pdfpages}
\usepackage{listings}
\usepackage{hyperref}
\usepackage{amsthm}
\usepackage{tikz}
\title{ }
\date{}
\author{}

\definecolor{mygreen}{rgb}{0,0.6,0}
\definecolor{mygray}{rgb}{0.5,0.5,0.5}
\definecolor{mymauve}{rgb}{0.58,0,0.82}

\numberwithin{equation}{section}
\theoremstyle{plain}
\newtheorem{thm}{Theorem}[section]

\newtheorem{obs}[thm]{Observation}
\newtheorem{cor}[thm]{Corollary}
\newtheorem{lem}[thm]{Lemma}
\theoremstyle{definition}

\newtheorem{defi}[thm]{Definition}

\title{RECONSTRUCTING TRIANGULATIONS OF $3$-MANIFOLDS FROM THEIR INTERSECTION MATRIX}
\date{}
\author{Jorge L. Arocha \\ email \href{mailto:arocha@matem.unam.mx}{arocha@matem.unam.mx} 
   \and Jorge Fernández-Hidalgo \\ email \href{mailto:jorgefernández@ciencias.unam.mx}{jorgefernandez@ciencias.unam.mx} }

\begin{document}

\maketitle

\abstract

The intersection matrix of a simplicial complex has entries equal to the rank of the intersection of its facets. In \cite{arocha} the authors prove the intersection matrix is enough to determine a triangulation of a surface up to isomorphism. In this work we show the intersection matrix is enough to determine the triangulation of a $3$-manifold up to isomorphism.

\section{Introduction}

\begin{defi}[simplicial complex]
	An abstract  simplicial complex is a pair $(V,\Delta)$ such that, $V$ is a finite set, $\Delta \subseteq 2^V$ and the following condition is satisfied, if  $X\in \Delta$ and $Y\subseteq X$  then $Y\in \Delta$. We call an element of $\Delta$ a simplex. If a simplex has $i$ elements we shall say it has rank $i$, dimension $i-1$ and call it an $i-1$-simplex. We denote the set of all $i$-simplices with $\Delta_i$. The maximal elements of $\Delta$ are called facets.
\end{defi}

\begin{defi}[isomorphism of simplicial complices]
	Given two simplicial complices $(V,\Delta)$ and $(V',\Delta')$ we say a map $\varphi: V\rightarrow V'$ is an isomorphism if $\varphi$ is bijective and $X\in \Delta$ if and only if $\varphi(X) \in \Delta'$.
\end{defi}

In this work we focus on simplicial complices with associated geometric realizations where the underlying space is a $3$-dimensional manifold (we call such objects triangulations of $3$-dimensional manifolds). For our purposes it is convenient to work with the underlying abstract simplicial complex and only use the geometric realization to give conditions the abstract simplicial complex must satisfy. One such condition is the following, given in terms of the neighbourhood of vertices.

\begin{defi}[neighbourhood of a vertex]
	Let $(V,\Delta)$ be a simplicial complex of dimension $3$ and let $v\in V$. We define the neighbourhood of $v$ as the simplicial complex of dimension $2$ spanned by the set of facets $\Delta'_{2} = \{ X - \{v\} | v \in X \in \Delta_3\}$.
\end{defi}

\begin{obs}
		\label{condvec}
		If $(V,\Delta)$ is a triangulation of  $3$-manifold the neighbourhood of $v$ must span a space homeomorphic to a sphere of dimension $2$.
\end{obs}

The intersection matrix of a simplicial complex is the matrix such that entry $i,j$ is equal to the rank of the intersection of the $i$'th facet and the $j$'th facet.

We use the following definition of an intersection preserving map which captures the notion of two simplicial complices having the same intersection matrix.

\begin{defi}[intersection preserving map]
	Given two simplicial complices $(V, \Delta)$ and $(V',\Delta')$ of dimension $3$ with corresponding sets of facets $\Delta_{3}$ and $\Delta'_{3}$ we say a map $f:\Delta_{3}\rightarrow \Delta'_{3}$ is an intersection preserving map if $f$ is bijective and $|X\cap Y| = |f(X)\cap f(Y)|$ for all $X,Y\in \Delta_{3}$.

	We say $f$ extends to an isomorphism if there exists an isomorphism $\varphi: V \rightarrow V'$ such that $\varphi(X) = f(X)$ for all $X\in \Delta_{3}$.
\end{defi}

The main result we shall prove is the following,

\begin{thm}
	\label{main}
	Let $(V,\Delta)$ and $(V',\Delta')$ be two triangulations of $3$-dimensional manifolds and $f:\Delta_3\rightarrow \Delta'_3$ an intersection preserving map, then $f$ extends to an isomorphism.
\end{thm}

Section $2$ proves some results about extensions of intersection preserving maps, and provides sufficient conditions for such an extension to exist. Sections $3$ and $4$ show these conditions hold in our case. Section $3$ is dedicated to introducing and classifyng $3$-cyclic shells, and in section $4$ we use this classification to obtain the desired result.

\input{extend.tex}

\input{shells.tex}

\input{condition.tex}

\input{biblio.tex}
\end{document}

%% file: extend.tex
\section{ Extending intersection preserving maps }

Let us first introduce a "reasonable" way of extending a map between facets to all simplices.

\begin{defi}
	Let $(V,\Delta)$ and $(V',\Delta')$ be simplicial complices of dimension $3$ and let $f:\Delta_{3} \rightarrow \Delta'_{3}$ be any map. We define $F:\Delta \rightarrow \Delta'$ via 

	\begin{equation*}
		F(X) = \smashoperator[r]{\bigcap \limits_{Y \in \Delta_{3} | X\subseteq Y} } f(Y).
	\end{equation*}

\end{defi}

\begin{obs}
	Let $(V,\Delta)$ be a triangulation of a $3$-manifold, then for all $v\in V$ the intersection of all facets containing $v$ is $\{v\}$.
\end{obs}

\begin{proof}
	Let $v\in V$, one has
	
	\begin{equation*}
		\smashoperator[lr]{\bigcap\limits_{X\in \Delta_{3} | v\in X} } X = \smashoperator[r]{\bigcap\limits_{X\in \Delta_{2} | v\in X}} X
	\end{equation*}

	This is because every triangle is contained in exactly two tetrahedra(since the simplicial complex is a triangulation), so every intersacand of the RHS is the intersection of two intersecands of the LHS.

	We can see

	\begin{equation*}
		\smashoperator[r]{\bigcap\limits_{X\in \Delta_{2} | v\in X}} X = \{v\}
	\end{equation*}

	by only considering triangles contained in one tetrahedra $X$.

\end{proof}

The following lemma shows $F$ can be the only extension of an intersection preserving map between $3$-manifolds.

\begin{lem}
	Let $(V,\Delta)$ and $(V',\Delta')$ be triangulations of $3$-manifolds and let $f:\Delta_{3} \rightarrow \Delta'_{3}$ be a bijection between facets which is extended by the isomorphism $\varphi: V \rightarrow V'$. Then for each $v \in V$ the set $F(\{v\})$ is equal to $\{\varphi(v)\}$.

\end{lem}

\begin{proof}
	Let $v\in V$. Since $f(X)= \varphi(X)$ for all $X\in \Delta_3$ we have,

	\begin{equation*}
		\smashoperator[r]{\bigcap\limits_{X\in \Delta_{3} | v \in X} }f(X) = \smashoperator[r]{\bigcap\limits_{X\in \Delta_{3} | v \in X} }\varphi(X).
	\end{equation*}

	Because $\varphi$ is bijective and $\smashoperator[r]{\bigcap\limits_{X\in \Delta_{3} | v \in X} } X = \{v\}$ we obtain,

	\begin{equation*}
		\smashoperator[r]{\bigcap\limits_{X\in \Delta_{3} | v \in X}} \varphi(X)  = \{\varphi(v)\}
	\end{equation*}

\end{proof}

We now show whenever this map and the inverse map can be defined, the resulting maps are isomorphisms.

\begin{thm}
	\label{primero}

	Let $(V,\Delta)$ and $(V',\Delta')$ be triangulations of $3$-manifolds and $f:\Delta_{3}\rightarrow \Delta'_{3}$ be a bijection between facets such that for all $v\in V$ and $v'\in V'$ the sets $F(\{v\})$ and $F^{-1}(\{v'\})$ are singletons. Then the map $\varphi:V\rightarrow V'$ defined by $\{ \varphi(v) \} = F(v)$ is an isomorphism and the map $\sigma:V'\rightarrow V$ defined by $\{\sigma(v')\} = F^{-1}(\{v'\})$ is its inverse.

\end{thm}

\begin{proof}
	We first show if $v\in X \in \Delta_{3}$ then $\varphi(v) \in f(X)$. One can see this by noticing $ \{ \varphi\{v\} \} = \smashoperator[r]{\bigcap\limits_{Y\in \Delta_{3} | v \in Y}} f(Y)  \subseteq f(X)$

	One can apply the previous argument once more (but in the opposite direction) using the element $\varphi(v) \in f(X) \in \Delta'_{3}$ to obtain  $\sigma(\varphi(v)) \in X$ for all $v\in X \in \Delta_{3}$.

	This implies $\sigma(\varphi(v)) \in \smashoperator[r]{\bigcap\limits_{X\in \Delta_{3} | v \in X}} X = \{v\}$. Therefore $\varphi$ is bijective and $\sigma$ is its inverse.

	It only remains to be proved that $\varphi(X) = f(X)$ for all $X \in \Delta_3$. We already have $\varphi(X) \subseteq f(X)$ and the equality is derived from the fact that $\varphi$ is injective and $X,f(X)$ are two finite sets of the same size.

\end{proof}

The following lemma will be used to prove theorem \ref{main}. In section $4$ we show the conditions of the lemma are met.

\begin{lem}
	\label{segundo}
		Let $(V,\Delta)$ and $(V',\Delta')$ be triangulations of $3$-manifolds and let $f: \Delta_{3} \rightarrow \Delta'_{3}$ be an intersection preserving map such that $F$ induces a bijection between $\Delta_{i}$ and $\Delta'_{i}$ for $1 \leq i \leq 3$. Then the map $\varphi: V \rightarrow V'$ given by $\{\varphi(v)\} = F(v)$ is well defined.

\end{lem}

	This proof will be slightly long and is composed of various lemmas.

	\begin{lem}

		\label{enemenosuno}

		For all $v\in V$ we have  $\smashoperator[r]{\bigcap\limits_{X \in \Delta_{3} | v \in X} } f(X) = \smashoperator[r]{\bigcap \limits_{X\in \Delta_{2} | v \in X} } F(X)$. 
	
	\end{lem}
	
	\begin{proof}

		For the $\supseteq$ containment we notice if $X\in \Delta_{3}$ contains $v$ then there is $Y\in \Delta_{2}$ containing $v$,such that $Y\subseteq X$, it follows $f(Y) \subseteq F(X)$.

	For the $\subseteq$ containment we use \ref{segundo} to see every intersecand in the RHS is the intersection of two  elements of the LHS.

	\end{proof}

	\begin{lem}
		\label{single1}
		Let $v\in V$ and  $X \in \Delta_3$ with $v\in X$. The set $\smashoperator[r]{\bigcap\limits_{Y\in \Delta_{2} | v \in Y \subseteq X} } F(Y)$ contains exactly one element.
	\end{lem}

	\begin{proof}
		This is the intersection of $3$ triangles contained in the tetrahedra $F(X)$. To see this is the case we note if $Y$ contains $X$ then $F(Y)$ is contained in $F(X)$, and because $F$ is a bijection between triangles we can be sure there are $3$ distinct intersecands.

	\end{proof}

	\begin{lem}
		\label{single2}
		Let $v\in V$ and $X\in \Delta_{2}$ with $v\in X$. The set $\smashoperator[r]{\bigcap\limits_{Y\in \Delta_{1} | v\in Y \subseteq X} }F(Y)$ contains exactly one element.
	\end{lem}
	\begin{proof}
		This is the intersection of $2$ edges contained in the triangle $F(X)$. To see this is the case we note if $Y$ contains $X$ then $F(Y)$ is contained in $F(X)$, and because $F$ is a bijection between edges we can be sure there are $2$ distinct intersecands.

	\end{proof}

	\begin{lem}
		\label{intermedio}
		Let $v\in V$, $X\in \Delta_{3}$ and $Y\in \Delta_{2}$ with $v\in Y \subseteq X$, then we have

		\begin{equation}
			\label{ecua}
			\smashoperator[r]{\bigcap\limits_{Z\in \Delta_{2} | v\in Z \subseteq X} }F(Z) = \smashoperator[r]{\bigcap\limits_{Z\in \Delta_{1} | v\in Z \subseteq Y} }F(Z)
		\end{equation}
	\end{lem}

	\begin{proof}

		We first show if $Z\in \Delta_{1}$ and $v\in Z \subseteq Y$ then,

		\begin{equation*}
			F(Z) = \smashoperator[r]{\bigcap\limits_{W\in \Delta_{2} | Z \subseteq W \subseteq X} } F(W)
		\end{equation*}

		To see this we first notice the $\subseteq$ relation holds, because $F(Z) \subseteq F(W)$ for both triangles $W$ in the intersection. To see equality holds we notice there are two distinct intersecands in the right hand side, so the intersection is an edge.

		We have thus shown every intersecand in the RHS of \ref{ecua} is the intersection of two intersecands of the LHS. It follows we have the $\subseteq$ relation in \ref{ecua}. Because both sets have exactly one element (by \ref{single1} and \ref{single2}) it follows they are equal.

	\end{proof}

	\begin{cor}
		If $X$ and $Y$ are two tetrahedra and $|X\cap Y|=3$ then
		\begin{equation*}
			\smashoperator[r]{\bigcap\limits_{Z\in \Delta_{2} | v \in Z \subseteq X}} F(Z) = \smashoperator[r]{\bigcap\limits_{Z\in \Delta_{2} | v \in Z \subseteq Y} } F(Z)
		\end{equation*}
	\end{cor}

	\begin{proof}
		Use \ref{intermedio} two times with the triangle $X\cap Y$.
	\end{proof}

	Because of \ref{condvec} we have the following observation.

	\begin{obs}
		Let $X,Y$ be two tetrahedra containing $v$. Then there is a finite sequence of tetrahedra $X = X_0,\dots X_{k-1} = Y$ such that $|X_i\cap X_{i+1}| = 3$ for all $0\leq i < k-1$ and $v\in X_i$ for all $1\leq i \leq k-1$.
	\end{obs}
	
	\begin{cor}
		\label{lastone}
		If $X$ and $Y$ are two tetrahedra containing $v$ then 

		\begin{equation*}
		\bigcap\limits_{Z\in \Delta_{2} | v \in Z \subseteq X} F(Z) = \bigcap\limits_{Z\in \Delta_{2} | v \in Z \subseteq Y} F(Z)
		\end{equation*}
	\end{cor}

	We are now finally able to prove $F(\{v\})$ is a singleton for each $v\in V$. We must only verify the following sequence of equalities, where the first equality is given by \ref{enemenosuno}.

		\begin{equation*}
			F(\{v\}) = \smashoperator{ \bigcap\limits_{X\in \Delta_{2} | v \in X} } F(X) =  \smashoperator[l]{\bigcap\limits_{X \in \Delta_3 | v\in X}} (\smashoperator[r]{\bigcap\limits_{Y\in \Delta_{2} | v \in Y \subseteq X} } F(Y) ).
		\end{equation*}

		Finally we note by \ref{lastone} this last intersection contains only one distinct set, which has only one element. This completes the proof of \ref{segundo}.

%% file: shells.tex
\section{3-Cyclic Shells}

The main goal of this section is to classify all $3$-cyclic shells. We begin with the definition of an $n$-cyclic shell. 

\begin{defi}[$n$-cyclic shell]

	Let $n\geq 2$. An $n$-cyclic shell is a simplicial complex of dimension $n$ with $k\geq3$ facets $\{H_0,H_1\dots H_{k-1}\}$, such that for all $0\leq i < j\leq k-1$ we have $|T_i\cap T_j|=n$ if $j=i+1$ or $i=0,j=k-1$, and $|T_i\cap T_j|=n-1$ otherwise.
\end{defi}

\begin{obs}
	There is a special family of $n$-cyclic shells we denote $nCW_k$, these contain facets $H_i = X \cup \{v_i,v_{i+1} \}$ for $0\leq i \leq k-2$ and $H_{k-1} = X \cup \{v_{k-1},v_0\}$, where $X$ is a $n-2$ simplex. We denote such $n$-cyclic shells by $nCW_k$. One can see the $3$-cyclic shells are precisely the simplicial complices with the same intersection matrices as $3CW_k$. \\

	In particular we note the neighbourhood of an edge $\{t,b\}$ in the triangulation of a $3$-manifold is spanned by facets $H_i=\{v_{i},v_{i+1},t,b\}$ for $0\leq i \leq k-2$ and $H_{k-1}=\{v_{k-1},v_0,t,b\}$. \end{obs}

In order to classify all $3$-cyclic shells we follow a procedure similar to the one followed to classify $2$-cyclic shells in \cite{arocha} (note these objects are simply called cyclic shells in \cite{arocha}). With this in mind we introduce the $3$-lineal shells.

\begin{defi}[$n$-lineal shell]
	An $n$-lineal shell is a simplicial complex of dimension $n$ spanned by a set of $k\geq 2$ facets $\{H_0,H_1\dots H_{k-1}\}$ such that for all $0\leq i < j \leq k-1$ we have $|T_i\cap T_j|=n$ if $j=i+1$ and $|T_i\cap T_j|=n-1$ otherwise.
\end{defi}

There is an infinite family of $n$-lineal shells which we denote $nLW_k$, these are spanned by $k$ facets $H_i=\{v_{i},v_{i+1}\} \cup X$ for $0\leq i\leq k-1$, where $X$ is an $n-2$-simplex not containing any of the $v_i$. Notice that if $H_0,H_1,\dots H_{k-1}$ spans an $n$-lineal shell then for every $0\leq i < j \leq k-1$ the facets $H_i,H_{i+1}\dots H_j$ also span an $n$-lineal shell. This allows us to classify the $3$-lineal shells starting from the small ones and gluing new facets one by one onto one of the $3$ "unused" faces of $H_0$ or $H_{k-1}$. In doing so we will need to deal with a large number of cases, but the arguments used in each case are simple.\\

Given a $2$-cyclic shell $(V,\Delta)$ and a new vertex $u$ not in $V$ we can consider the simplicial complex $(V\cup \{u\},\Delta')$ of dimension $3$ in which the facets are $\Delta'_{3} = \{ X \cup \{u\} | X\in \Delta_2\}$.

One can see $(V\cup \{u\}, \Delta')$ is a $3$-cyclic shell, we refer to this process as lifting a $2$-shell. We denote a lifted $2$-shell by using the same name as in \cite{arocha} with a $3$ to the left. We now make the definition of a "lifted" $3$-lineal and $3$-cyclic shell precise.

\begin{defi}
	We say a $3$-lineal shell $(V,\Delta)$ (respectively $3$-cyclic shell) is a lifted $3$-lineal shell (respectively $3$-cyclic shell) if there exists a vertex $\alpha$ contained in all its facets. In this case we notice the $2$-dimensional simplicial complex spanned by all the $2$-simplices that do not contain $\alpha$ is a $2$-lineal shell (respectively $2$-cyclic shell), we say $(V,\Delta)$ is obtained by lifitng this $2$-lineal shell (respectively $2$-cyclic shell).

\end{defi}

We now include the classification of $2$-cyclic and $2$-lineal shells found in \cite{arocha}. The classification of $2$-shells (and thus of lifted $3$-shells) will reduce computations when we classify the $3$-shells.

\begin{thm}
	The lineal $2$-shells are $2LW_k$ as well as the following,\\

\begin{tabular}{ c | c | c }
	name & vertex set & facet set \\ \hline
	$LE_4$ & $\{v,a,b,b',a'\}$  & $\{v,a,b\}, \{a,b,b'\}, \{a',b,b'\},\{v,a',b'\}$ \\ \hline
	$LE_5$ & same as $LE_4$ and $c$  & same as $LE_4$ and $\{b',v,c\}$ \\ \hline
	$LE_6$ & same as $LE_5$ and $d$  & same as $LE_5$ and $\{v,d,b\}$ \\ \hline
\end{tabular}

\end{thm}

\begin{thm}

The cyclic $2$-shells are $2CW_k$ as well as the following,\\

\begin{tabular}{ c | c | c }
	name & vertex set & facet set \\ \hline
	$CE_5$ & same as $LE_4$  &  same as $LE_4$ and $\{v,a,a'\}$ \\ \hline
	$CE_6$ & same as $LE_5$  &  same as $LE_5$ and $\{v,b,c\}$ \\ \hline
\end{tabular}

\end{thm}

	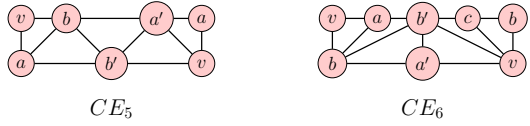
\begin{figure}[h]
		\centering{
		\scalebox{0.6}{
			\begin{tikzpicture}[main_node/.style={circle,fill=red!20,draw,minimum size=1em,inner sep=3pt]}]
				\node[main_node] (v) at (0,1) {$v$};
				\node[main_node] (a) at (0,0) {$a$};
				\node[main_node] (b) at (1,1) {$b$};
				\node[main_node] (bp) at (2,0) {$b'$};
				\node[main_node] (ap) at (3,1) {$a'$};
				\node[main_node] (v2) at (4,0) {$v$};
				\node[main_node] (a2) at (4,1) {$a$};
				\path[draw,thick] (a) edge (v);
				\path[draw,thick] (a) edge (b);
				\path[draw,thick] (b) edge (v);
				\path[draw,thick] (a) edge (bp);
				\path[draw,thick] (b) edge (bp);
				\path[draw,thick] (b) edge (ap);
				\path[draw,thick] (bp) edge (ap);
				\path[draw,thick] (bp) edge (v2);
				\path[draw,thick] (ap) edge (v2);
				\path[draw,thick] (ap) edge (a2);
				\path[draw,thick] (v2) edge (a2);
				\node [label={[font=\large] below:{$CE_5$ } }] (*) at (2,-0.5  ) {};
			\end{tikzpicture}
			\hspace{2cm}
			\begin{tikzpicture}[main_node/.style={circle,fill=red!20,draw,minimum size=1em,inner sep=3pt]}]
				\node[main_node] (v) at (0,1) {$v$};
				\node[main_node] (b) at (0,0) {$b$};
				\node[main_node] (a) at (1,1) {$a$};
				\node[main_node] (bp) at (2,1) {$b'$};
				\node[main_node] (ap) at (2,0) {$a'$};
				\node[main_node] (c) at (3,1) {$c$};
				\node[main_node] (b2) at (4,1) {$b$};
				\node[main_node] (v2) at (4,0) {$v$};
				\path[draw,thick] (a) edge (v);
				\path[draw,thick] (a) edge (b);
				\path[draw,thick] (b) edge (v);
				\path[draw,thick] (a) edge (bp);
				\path[draw,thick] (b) edge (bp);
				\path[draw,thick] (b) edge (ap);
				\path[draw,thick] (bp) edge (ap);
				\path[draw,thick] (bp) edge (v2);
				\path[draw,thick] (ap) edge (v2);
				\path[draw,thick] (b2) edge (c);
				\path[draw,thick] (bp) edge (c);
				\path[draw,thick] (v2) edge (b2);
				\path[draw,thick] (c) edge (v2);
				\node [label={[font=\large] below:{$CE_6$ } }] (*) at (2,-0.5  ) {};
			\end{tikzpicture}
			}
			}
			\caption{ The non-orientable simplicial complices $CE_5$ and $CE_6$}
	\end{figure}

	\vspace{1cm}

	We now begin with the classification of $3$-lineal shells up to isomorfism. Easy arguments show that $3LW_2$ and $3LW_3$ are the only $3$-lineal shells with $2$ and $3$ facets. We continue by analyzing the cases with $4$ facets.

	\subsection{4 facets}

	Every $3$-cyclic shell with $4$ facets can be obtained by gluing a new facet onto $3LW_3$. Because of symmetry we can assume it shares $3$ vertices with $H_2=\{v_2,v_3,t,b\}$ and contains $t$. There are $2$ cases (from here on out $\alpha$ will always be the symbol used for the new vertex, we will initially assume $\alpha$ is a new vertex, and then consider the possibility of identifying $\alpha$ with a preexisting vertex, we will also denote the new facet by $H_k$ or $H_{-1}$ depending on whether it shares $3$ vertices with $H_{k-1}$ or $H_0$) :\\

	$H_3=\{v_3,t,b,\alpha\}$, this yields $3LW_4$\\

	$H_3=\{v_2,v_3,t,\alpha\}$, we must identify $\alpha$ with a vertex of $H_0$, if $\alpha=v_1$ or $\alpha=b$ then $|H_3\cap H_1|=1$, so  $\alpha=v_0$, this yields a $3$-lineal shell we shall denote $3LE_4$.\\

	\subsection{5 facets}

	First we try to extend $3LW_4$. Because of symmetry  we can assume $H_4$ shares $3$ vertices with $H_3=\{v_3,v_4,t,b\}$ and contains $t$, there are $2$ cases:\\

	$H_4=\{v_4,t,b,\alpha\}$, this yields $3LW_5$\\

	$H_4=\{v_3,v_4,t,\alpha\}$, we need $|H_4\cap H_0|=2$ and $|H_4\cap H_1|=2$ so $\alpha$ must be equal to the only common vertex of $H_0$ and $H_1$ that isn't $t$ or $b$, that is $\alpha=v_1$. This yields a $3$-lineal shell we shall denote $3LE_5$\\

	We now try to extend $3LE_4$, we will relable it to have vertices $\{v,a,b,b',a',z\}$ and facets $H_0=\{v,a,b,z\}$,$H_1=\{a,b,b',z\},$\ $
	H_2=\{b,b',a',z\}$,$H_3=\{b',a',v,z\}$. Notice the permutation $(a,a')(b,b')$ is an automorphism, so we can assume $H_4$ shares $3$ vertices with $\{b',a',v,z\}$, there are $3$ cases:\\

	$H_4=\{v,z,b',\alpha\}$, yields $3LE_5$ ( to see this notice $H_1,H_2,H_3,H_4$ span a complex isomorphic to $3LW_4$, so this shell already appeared previously).\\

	$H_4=\{v,a',z,\alpha\}$, the rank of every intersection is correct except for $H_1\cap H_4$, so we must identify $\alpha$  with a vertex exclusive to $H_1$, none exist.\\

	$H_4=\{v,a',b',\alpha\}$, We must have $|H_0\cap H_4|=2$ and $|H_1\cap H_4|=2$, so $\alpha=a,b$ or $z$. The last two options imply $|H_4,H_2|=3$. So we must have $\alpha=a$, this yields a $3$-lineal shell we shall denote $LS_5$.\\

	\begin{lem}

		If a $3$-linear shell with $k$ facets contains $3LW_5$ then it is $3LW_k$

	\end{lem}

	\begin{proof}

		\label{trunca}

		We just have to prove $3LW_k$ can only be extended to $3LW_{k+1}$ for $k\geq 5$. In such an extension we can assume $H_k$ shares $3$ vertices with $\{v_{k-1},v_k,t,b\}$ and $t\in H_k$ There are two cases:

		$H_k=\{v_k,t,b,\alpha\}$, this yields $3LW_{k+1}$\\

		$H_k=\{v_{k-1},v_k,t,\alpha\}$, we have $|H_k\cap H_0|=1$ and $H_k\cap H_2|=1$, so $z$ must be $t$ or $b$, which is impossible.\\

	\end{proof}

	\subsection{6 facets}

	First we try to extend $3LE_5$, we present it with vertices $\{v,a,a',b,b',z,c\}$ and facets $H_0=\{v,a,b,z\},H_1=\{a,b,b',z\},$ \ $H_2=\{b,b',a',z\},H_3=\{b',a',v,z\},H_4=\{b',v,c,z\}$. Notice that if $z$ is in the new facet then the complex must be a lifted $3$-lineal shell, so this would only yield $3LE_6$, there are only two other options:\\

	$H_{-1}=\{v,a,b,\alpha\}$, notice $|H_{-1}\cap H_2|=|H_{-}\cap H_4|=1$, so $\alpha=b'$ or $\alpha=z$,the first option forces $|H_{-1}\cap H_1|=3$ and the second $H_{-1}=H_0$.\\

	$H_{5}=\{b',v,c,\alpha\}$, notice $|H_5\cap H_0|=|H_5\cap H_2|=1$, so $\alpha=b$ or $\alpha=z$, the second option forces $H_5=H_4$ and the first yields a $3$-lineal shell we shall denote $LF_6$\\

	We now try to extend $LS_5$, we present it with vertices $\{v_0,v_1,v_2,v_3,t,b\}$ and facets $H_0=\{v_0,v_1,v_3,t\},H_1=\{v_0,v_1,t,b\},H_2=\{v_1,v_2,t,b\},H_3=\{v_2,v_3,t,b\},H_4=\{v_0,v_2,v_3,b\}$, since the permutation $(v_0,v_3)(v_2,v_3)(t,b)$ is an automorphism we can assume $H_6$ shares $3$ vertices with $H_4$. There are $3$ cases:

	$H_{5}=\{v_0,v_2,v_3,\alpha\}$,  every intersection has correct rank except $H_5\cap H_1$ and $H_5\cap H_2$, so $\alpha$ must be a vertex exclusive to $H_1$ and $H_2$, none exist.\\

	$H_5=\{v_0,v_2,b\}$, every intersection has correct rank except $H_5\cap H_0$, so $\alpha$ must be a vertex exclusive to $H_0$, none exist.\\

	$H_5=\{v_0,v_3,b\}$, every intersection has correct rank except $H_5\cap H_2$, so $\alpha$ must be a vertex exclusive to $H_0$, none exist\\

	\subsection{7 facets}

	We first try to extend $3LE_6$, we present it with vertices $\{d,v,b,a,b',a',c,z\}$ and facets $H_0=\{d,v,b,z\}$, $H_1=\{v,b,a,z\}$, $H_2=\{a,b,b',z\}$, $H_3=\{b,b',a',z\}$, $H_4=\{b',a',v,z\}$, $H_5=\{b',v,c,z\}$, if $z$ is in the new facet then the complex must be a lifted $3$ lineal shell. Notice the permutation $(a,a')(b,b')(c,d)$ is an automorphism of $3LE_6$, so there is only one case:\\

	$H_{6}=\{b',v,c,\alpha\}$, notice $|H_{5}\cap H_0|=|H_5\cap H_2|=1$, so $\alpha=b$, this yields a $3$-lineal shell we shall denote $LF_7$\\

	We now try to extend $LF_6$, we present it with vertices $\{a_0,a_1,a_2,a_3,a_4,a_5,a_6\}$ and facets $H_0=\{a_0,a_1,a_2,a_6\},H_1=\{a_0,a_1,a_2,a_4\},H_2=\{a_2,a_3,a_4,a_0\},H_3=\{a_2,a_3,a_4,a_6\},H_4=\{a_4,a_5,a_6,a_2\},H_5=\{a_4,a_5,a_6,a_0\}$. Notice the permutation $(0,6)(1,5),(2,4)$ is an automorphism, so there are $3$ cases:\\

	$H_6=\{a_5,a_6,a_0,\alpha\}$, notice $|H_6\cap H_1|=|H_6\cap H_3|=1$, so $\alpha=a_2$ or $a_4$, in either case we would have $|H_6\cap H_4|=3$.\\

	$H_6=\{a_4,a_6,a_0,\alpha\}$, this yields a $3$-lineal shell , we can see it is isomorphic to $LF_7$ via the isomorphism $a_0\mapsto v,a_1\mapsto c, a_2\mapsto b',a_3 \mapsto a', a_4\mapsto z, a_5\mapsto a, a_6\mapsto b, \alpha \mapsto d$.\\

	$H_6=\{a_4,a_5,a_0,\alpha\}$, notice $|H_6\cap H_0|=|H_6\cap H_3|=1$, so $\alpha=a_2$ or $a_6$, in either case we would have $|H_6\cap H_4|=3$.\\

	\subsection{8 facets}

	We now try to extend $LF_7$, we present it with vertices $\{a_0,a_1,a_2,a_3,a_4,a_5,a_6,a_7\}$ and facets $H_0=\{a_0,a_1,a_2,a_6\},H_1=\{a_0,a_1,a_2,a_4\},H_2=\{a_2,a_3,a_4,a_0\},H_3=\{a_2,a_3,a_4,a_6\},H_4=\{a_4,a_5,a_6,a_2\},H_5=\{a_4,a_5,a_6,a_0\},H_6=\{a_4,a_6,a_0,a_7\}$. 

	If we add a facet $H_{-1}$ it must be $\{a_2,a_0,a_6,\alpha\}$  with $\alpha$ a new vertex (because this was the only way to extend $LF_6$ ). This works and we call this $3$-lineal shell $LF_8$.

	There are $3$ options for $H_7$:

	$H_7=\{a_6,a_0,a_7,\alpha \}$, notice $|H_7\cap H_1|=|H_7\cap H_3|=1$, so $\alpha=a_2$ or $a_4$, if $\alpha=a_2$ then $|H_0\cap H_7|=3$ and if $\alpha=a_4$ then $H_7=H_6$\\

	$H_7=\{a_4,a_0,a_7,\alpha \}$, notice $|H_7\cap H_0|=|H_7\cap H_3|=1$, so $\alpha=a_2$ or $a_6$, if $\alpha=a_2$ then $|H_1\cap H_1|=3$ and if $\alpha=a_6$ then $H_7=H_6$\\

	$H_7=\{a_4,a_6,a_7,\alpha \}$, notice $|H_7\cap H_0|=|H_7\cap H_2|=1$ so $\alpha=a_0$ or $a_2$, if $\alpha=a_0$ then $H_7=H_6$ and if $\alpha=a_2$ then $|H_7\cap H_2|=3$\\

	\subsection{9 facets}

	We present $LF_8$ with vertices $\{a_{-1},a_0,a_1,a_2,a_3,a_4,a_5,a_6,a_7\}$ and facets $H_0=\{a_2,a_0,a_6,a_{-1}\}$,$H_1=\{a_0,a_1,a_2,a_6\},H_2=\{a_0,a_1,a_2,a_4\},H_3=\{a_2,a_3,a_4,a_0\},H_4=\{a_2,a_3,a_4,a_6\},H_5=\{a_4,a_5,a_6,a_2\},H_6=\{a_4,a_5,a_6,a_0\},H_7=\{a_4,a_6,a_0,a_7\}$. Notice the bijection $(a_{-1},a_7)(a_0,a_6)(a_1,a_5)(a_2,a_4)$ is an isomorphism. So no new facet can be added (because we already showed no new facet can be added to the other end of $LF_7$) . We have now proved the following theorem:\\

	\begin{thm}
	The $3$-lineal shells are $3LE_4,3LE_5,3LE_6,LS_5,LF_6,LF_7,LF_8$ and $3LW_k$ with $k\geq 2$.\\
	\end{thm}

	\subsection{closing the $3$-lineal shells}

	In order to find all $3$-cyclic shells we must analyze all possible ways to add a facet $H_k$ to a $3$-lineal shell $H_0,H_1,\dots  H_{k-1}$ so that the facets $H_0,H_1\dots H_k$ span a $3$-cyclic shell. In such a situation it is clear that both vertices of $H_0\cap H_{k-1}$ must be contained in $H_k$, as well as one vertex from $H_0\setminus H_{k-1}$ and one from $H_{k-1}\setminus H_0$. So we have reduced the possibilities for $H_k$ to $4$ facets. This observation clearly also implies that if a lifted $3$-lineal shell is closed, this results in a lifted $3$-cyclic shell. So we must only consider closing the following $3$-lineal shells: $LS_5,LF_6,LF_7,LF_8,LG_7,LG_8$. This process just consists of revising each of the $4\times 6$ possibilities.\\

	\textbf{$LS_5$:} we present it with vertices $\{v_0,v_1,v_2,v_3,t,b\}$ and facets $H_0=\{v_0,v_1,v_3,t\},H_1=\{v_0,v_1,t,b\},H_2=\{v_1,v_2,t,b\},H_3=\{v_2,v_3,t,b\},H_4=\{v_0,v_2,v_3,b\}$\\

	\begin{tabular}{ c | c }
		$H_5$ & incorrect intersection \\ \hline
		$\{v_0,v_3,v_1,v_2\}$ & none \\ \hline
		$\{v_0,v_3,v_1,b \}$	& $|H_5\cap H_1|=3$ \\ \hline
		$\{v_0,v_3,t,v_2 \}$	& $|H_5\cap H_3|=3$ \\ \hline
		$\{v_0,v_3,t,b\}$ &  $|H_5\cap H_1|=3$  \\ \hline
	\end{tabular}
	\vspace{.5 cm}

	We denote the resulting shell $CS_6$.\\

	\noindent \textbf{$LF_6$:} we present it with vertices $\{a_0,a_1,a_2,a_3,a_4,a_5,a_6\}$ and facets $H_0=\{a_0,a_1,a_2,a_6\},H_1=\{a_0,a_1,a_2,a_4\},H_2=\{a_2,a_3,a_4,a_0\},H_3=\{a_2,a_3,a_4,a_6\},H_4=\{a_4,a_5,a_6,a_2\},H_5=\{a_4,a_5,a_6,a_0\}$

	\begin{tabular}{ c | c }
		$H_6$ & incorrect intersection \\ \hline
		$\{a_0,a_6,a_1,a_4 \}$ & $|H_6\cap H_1|=3$ \\ \hline
		$\{a_0,a_6,a_1,a_5 \}$	& $|H_6\cap H_2|=1$ \\ \hline
		$\{a_0,a_6,a_2,a_4 \}$	& $|H_5\cap H_1|=3$ \\ \hline
		$\{a_0,a_6,a_2,a_5 \}$ &  $|H_5\cap H_4|=3$  \\ \hline
	\end{tabular}
	\vspace{.5 cm}

	\noindent \textbf{$LF_7$:} we present it with vertices $\{a_0,a_1,a_2,a_3,a_4,a_5,a_6,a_7\}$ and facets $H_0=\{a_0,a_1,a_2,a_6\},H_1=\{a_0,a_1,a_2,a_4\},H_2=\{a_2,a_3,a_4,a_0\},H_3=\{a_2,a_3,a_4,a_6\},H_4=\{a_4,a_5,a_6,a_2\},H_5=\{a_4,a_5,a_6,a_0\},H_6=\{a_4,a_6,a_0,a_7\}$\\

	\begin{tabular}{ c | c }
		$H_7$ & incorrect intersection \\ \hline
		$\{a_0,a_6,a_1,a_4 \}$ & $|H_7\cap H_1|=3$ \\ \hline
		$\{a_0,a_6,a_1,a_7 \}$	& $|H_7\cap H_2|=1$ \\ \hline
		$\{a_0,a_6,a_2,a_4 \}$	& $|H_7\cap H_5|=3$ \\ \hline
		$\{a_0,a_6,a_2,a_7 \}$ &  none  \\ \hline
	\end{tabular}
	\vspace{.5 cm}

	We denote the resulting shell $CF_8$.\\

	\textbf{$LF_8$:} we present it with vertices $\{a_{-1},a_0,a_1,a_2,a_3,a_4,a_5,a_6,a_7\}$ and facets $H_0=\{a_2,a_0,a_6,a_{-1}\}$,$H_1=\{a_0,a_1,a_2,a_6\},H_2=\{a_0,a_1,a_2,a_4\},H_3=\{a_2,a_3,a_4,a_0\},H_4=\{a_2,a_3,a_4,a_6\},H_5=\{a_4,a_5,a_6,a_2\},H_6=\{a_4,a_5,a_6,a_0\},H_7=\{a_4,a_6,a_0,a_7\}$\\

	\begin{tabular}{ c | c }
		$H_8$ & incorrect intersection \\ \hline
		$\{a_0,a_6,a_2,a_4 \}$ & $|H_8\cap H_1|=3$ \\ \hline
		$\{a_0,a_6,a_2,a_7 \}$	& $|H_8\cap H_1|=3$ \\ \hline
		$\{a_0,a_6,a_{-1},a_4 \}$	& $|H_8\cap H_6|=3$ \\ \hline
		$\{a_0,a_6,a_{-1},a_7 \}$ &  $|H_7\cap H_2|=1$  \\ \hline
	\end{tabular}
	\vspace{.5 cm}

	\begin{thm}
		The $3$-cyclic shells are $3CE_5,3CE_6,CS_6,CF_8$ and $3CW_k$ for $k\geq 3$.\\
	\end{thm}

%% file: condition.tex
\section{ Intersection preserving maps send edges to edges }

We devote this section to proving the following theorem:\\

\begin{thm}
	\label{conditions}
	Let $(V,\Delta)$ and $(V',\Delta')$ be triangulations of $3$-manifolds. If $f:\Delta_{3}\rightarrow \Delta'_3$ is an intersection preserving map and the facets $\{H_0,H_1\dots H_{k-1}\}$ span $3CW_k$ then $\{f(H_0),f(H_1)\dots f_(H_{k-1})\}$ span $3CW_k$.\\
\end{thm}

\begin{obs}
		If $(V,\Delta)$ is a triangulation of a $3$-manifold it cannot contain  $3E_5$ or $3E_6$, as the neighbourhood of the lifting vertex would contain a copy of $E_5$ or $E_6$, this is impossible as these are non-orientable and would contradict \ref{condvec}.\\
\end{obs}

With this in mind we will only have to prove two lemmas concerning $CS_6$ and $CF_8$ to complete the proof of \ref{conditions}. Surprisingly the proof of both lemmas are rather similar.

\begin{lem}

	If $f:(V,\Delta)\rightarrow (V',\Delta')$ is an intersection preserving map between triangulations of $3$-manifolds and the facets $\{H_0,H_1,H_2,H_3,H_4,H_5\}$ span $3CW_6$ then $\{f(H_0),f(H_1)\dots f(H_{n-1})\}$ cannot span $CS_6$.\\

\end{lem}

\begin{proof}
	Assume it does, let $H_0=\{a_0,a_1,t,b\},H_1=\{a_1,a_2,t,b\},H_2=\{a_2,a_3,t,b\},H_3=\{a_3,a_4,t,b\},H_4=\{a_4,a_5,t,b\},H_5=\{a_5,a_0,t,b\}$ and $T_0=\{b_0,b_1,b_2,b_3\},T_1=\{b_1,b_2,b_3,b_4,\},T_2=\{b_2,b_3,b_4,b_5\},T_3=\{b_3,b_4,b_5,b_0\},T_4=\{b_4,b_5,b_0,b_1\},T_5=\{b_5,b_0,b_1,b_2\}$ such that $F(H_i)=T_i$ for all $0\leq i < 6$.\\

	\begin{figure}[h]
	\foreach \n  in {0,1,...,5} {
		\scalebox{0.6}{
				\begin{tikzpicture}[main_node/.style={circle,fill=red!20,draw,minimum size=2em,inner sep=3pt]}]
					\foreach \x in {0,1,...,5}
					\node[main_node] (\x) at (135-360/6*\x:2) {$b_{\x}$};
					\pgfmathsetmacro{\z}{int(mod(\n,6))}
					\pgfmathsetmacro{\f}{int(mod(\n+1,6))}
					\pgfmathsetmacro{\s}{int(mod(\n+2,6))}
					\pgfmathsetmacro{\t}{int(mod(\n+3,6))}
					\path[draw,thick] (\z) edge (\f);
					\path[draw,thick] (\z) edge (\s);
					\path[draw,thick] (\z) edge (\t);
					\path[draw,thick] (\f) edge (\s);
					\path[draw,thick] (\f) edge (\t);
					\path[draw,thick] (\s) edge (\t);
					\node [label={[font=\huge] below:{$T_\n$ } }] (*) at (0,-3  ) {};
				\end{tikzpicture}
				}
				}
				\caption{ The facets $T_i$ }
	\end{figure}
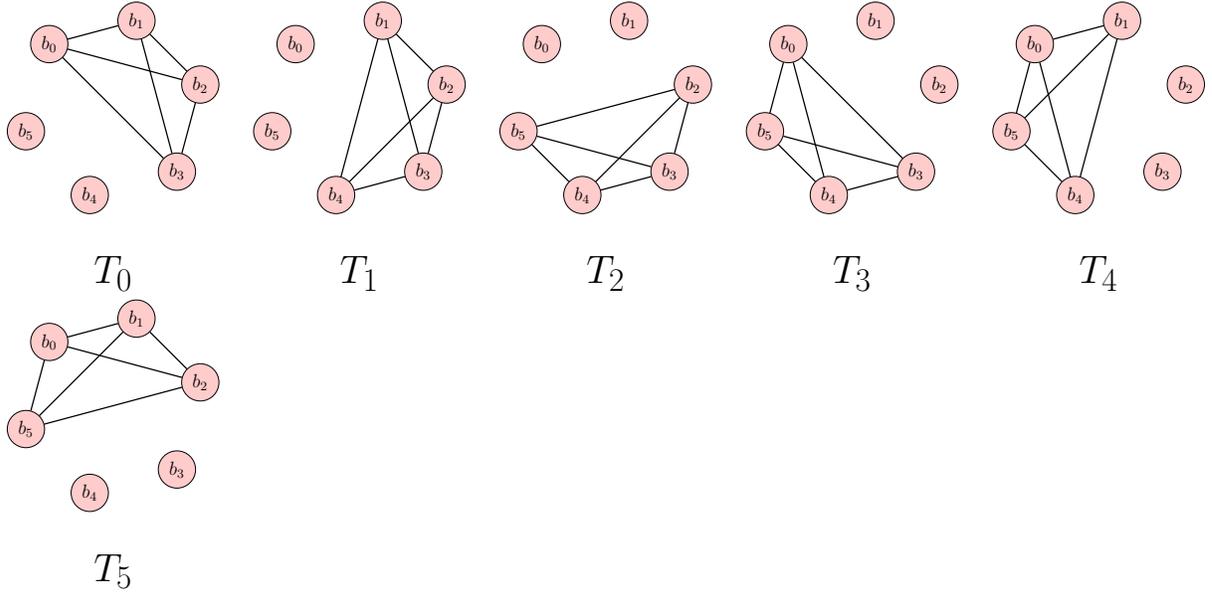

	\noindent Let $J_0=\{a_0,a_1,t,\alpha\}$ be the only other facet of $(V,\Delta)$ with $\{a_0,a_1,t\}\subseteq J_0$. Notice that $\{b_0,b_1,b_3\}\subseteq f(J_0)$ or $\{b_0,b_2,b_3\}\subseteq f(J_0)$. This is because $f(J_0)$ must share three vertices with $T_0$ and cannot be one of the $T_i$, as $f$ is bijective. \\

	\noindent Case $1$: $\{b_0,b_1,b_3\} \subseteq f(J_0)$, then we have $|f(J_0)\cap T_3|,|f(J_0)\cap T_4|\geq 2$. This implies $|J_0\cap H_3|,|J_0\cap H_4|\geq 2$. Which is only possible if $\alpha=a_4$.\\

	\noindent Case $2$: $\{b_0,b_2,b_3\}\subseteq f(J_0)$, then we have $|f(J_0)\cap T_2|,|f(J_0)\cap T_3|\geq 2$. This implies $|J_0\cap H_2|,|J_0\cap H_3|\geq 2$. Which is only possible if $\alpha=a_3$.\\

	\noindent Hence $\alpha=a_3$ or $a_4$. We can define $J_i$ analogously for $0\leq i <6$. Explicitly $J_i$ is the only other facet containing $H_i - \{b\}$. Because of rotational simmetry we have the following observation.

	\begin{obs}
		$J_i=\{a_i,a_{i+1},t,a_{i+3}\}$ or $J_i=\{a_i,a_{i+1},t,a_{i+4}\}$ for $0\leq i < 6$ (here we are using the notation $a_i = a_{i-6}$ for $6\leq i < 12$).
	\end{obs}

	\begin{cor}

		\label{triauno}
		For each $0\leq i < 6$ at least one of the triangles $\{a_i,a_{i+1},a_{i+3}\}$ or $\{a_i,a_{i+1},a_{i+4}\}$ appears in the boundary of $t$.

	\end{cor}

	\noindent Now notice that the triangles $\{a_i,a_{i+1},b\}$ are all contained in the boundary of $t$ and span a surface homeomorphic to a disk whose boundary  is the cycle $a_0,a_1,a_2,a_3,a_4,a_5$. Therefore the remaining triangles in the boundary of $t$ must also span a disk with the same boundary.\\

	Applying \ref{triauno} with $i=0$ we get two cases,\\

	Case $1$: Triangle $\{a_0,a_1,a_3\}$ appears in the boundary of $t$. Using \ref{triauno} with $i=1$ we have at least one of $\{a_1,a_2,a_4\}$ and $\{a_1,a_2,a_5\}$ must appear. Notice that the edge $a_0a_3\in\{a_0,a_1,a_3\}$ crosses both $a_1a_4\in \{a_1,a_2,a_4\}$ and $a_1a_5\in \{a_1,a_2,a_5\}$, this is a contradiction, as the triangles would not form a disk with boundary $a_0,a_1,a_2,a_3,a_4,a_5$.\\

	Case $2$: Triangle $\{a_0,a_1,a_4\}$ appears in the boundary of $t$. Using \ref{triauno} with $i=2$ we have at least one of $\{a_2,a_3,a_5\}$ and $\{a_2,a_3,a_0\}$ must appear. Notice that the edge $a_1a_4\in\{a_0,a_1,a_4\}$ crosses both $a_3a_5\in \{a_2,a_3,a_5\}$ and $a_3a_0\in \{a_1,a_3,a_0 \}$, this is a contradiction, as the triangles would not form a disk with boundary $a_0,a_1,a_2,a_3,a_4,a_5$.\\

\end{proof}

\begin{lem}

	If $f:(V,\Delta)\rightarrow (V',\Delta')$ is an intersection preserving map between triangulations of $3$-manifolds and the facets $\{H_0,H_1,H_2,H_3,H_4,H_5,H_6,H_7\}$ span $3CW_8$ then $\{f(H_0),f(H_1)\dots f(H_{n-1})\}$ cannot span $CF_8$.\\

\end{lem}

\begin{proof}

	Assume it does, let $H_0=\{a_0,a_1,t,b\},H_1=\{a_1,a_2,t,b\},H_2=\{a_2,a_3,t,b\},H_3=\{a_3,a_4,t,b\},H_4=\{a_4,a_5,t,b\},H_6=\{a_6,a_7,t,b\},H_7=\{a_7,a_0,t,b\}$ and $T_0=\{b_0,b_1,b_2,b_6\},T_1=\{b_0,b_1,b_2,b_4\},T_2=\{b_2,b_3,b_4,b_0\},T_3=\{b_2,b_3,b_4,b_6\},T_4=\{b_4,b_5,b_6,b_2\},T_5=\{b_4,b_5,b_6,b_0\},T_6=\{b_6,b_7,b_0,b_4\},T_7=\{b_6,b_7,b_0,b_2\}$ such that $f(H_i)=T_i$ for all $0\leq i < 8$.\\

	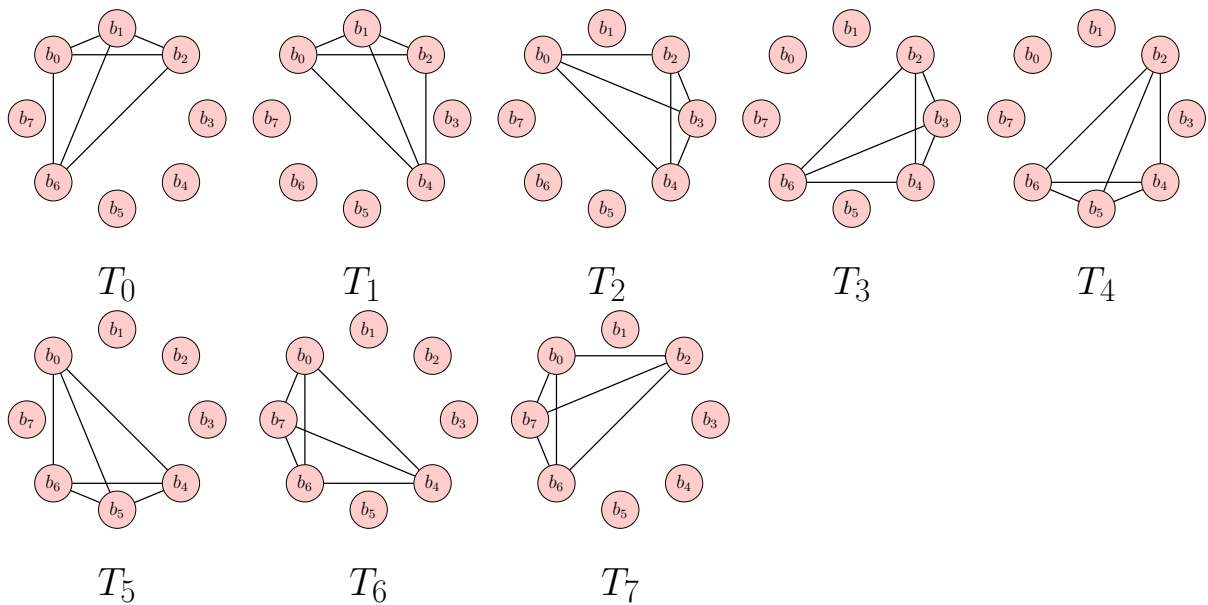
\begin{figure}[h]
		\foreach \n  in {0,1,...,7} {
			\scalebox{0.6}{
				\begin{tikzpicture}[main_node/.style={circle,fill=red!20,draw,minimum size=2em,inner sep=3pt]}]
					\foreach \x in {0,1,...,7}
					\node[main_node] (\x) at (135-360/8*\x:2) {$b_{\x}$};
					\pgfmathsetmacro{\l}{int(mod(\n+6,8))}
					\pgfmathsetmacro{\m}{int(mod(\n+7,8))}
					\pgfmathsetmacro{\o}{int(mod(\n+1,8))}
					\pgfmathsetmacro{\p}{int(mod(\n+2,8))}
					\pgfmathsetmacro{\v}{int(mod(\n+3,8))}
					\ifodd \n
					\path[draw,thick] (\n) edge (\m);
					\path[draw,thick] (\n) edge (\o);
					\path[draw,thick] (\n) edge (\v);
					\path[draw,thick] (\m) edge (\o);
					\path[draw,thick] (\m) edge (\v);
					\path[draw,thick] (\o) edge (\v);
					\else
					\path[draw,thick] (\n) edge (\o);
					\path[draw,thick] (\n) edge (\p);
					\path[draw,thick] (\n) edge (\l);
					\path[draw,thick] (\o) edge (\p);
					\path[draw,thick] (\o) edge (\l);
					\path[draw,thick] (\p) edge (\l);
					\fi
					\node [label={[font=\huge] below:{$T_\n$ } }] (*) at (0,-3  ) {};
				\end{tikzpicture}
				}
				}
				\caption{ The facets $T_i$ }
	\end{figure}

	\noindent Let $J_0=\{a_0,a_1,t,\alpha\}$ be the only other facet in $(V,\Delta)$ with $\{a_0,a_1,t\}\subseteq J_0$. Notice that $\{b_0,b_1,b_6\}\subseteq f(J_0)$ or $\{b_1,b_2,b_6\}\subseteq f(J_0)$. This is because $f(J_0)$ must share three vertices with $T_0$ and cannot be one of the $T_i$, as $f$ is bijective.\\

	\noindent Case $1$: $\{b_0,b_1,b_6\}\subseteq f(J_0)$, then we have $|f(J_0)\cap T_5|,|f(J_0)\cap T_6|\geq 2$. This implies $|J_0\cap H_5|,|J_0\cap H_6|\geq 2$. Which is only possible if $\alpha=a_6$.\\

	\noindent Case $2$: $\{b_1,b_2,b_6\}\subseteq f(J_0)$, then we have $|f(J_0)\cap T_3|,|f(J_0)\cap T_4|\geq 2$. This implies $|J_0\cap H_5|,|J_0\cap H_6|\geq 2$. Which is only possible if $\alpha=a_4$.\\

	\noindent Hence $\alpha=a_4$ or $a_6$. We can define $J_i$ analogously for $0\leq i <8$. Explicitly $J_i$ is the only other facet containing $H_i - \{b\}$. Notice that both complices have the dihedral symmetries of the squares $a_0,a_2,a_4,a_6$ and $b_0,b_2,b_4,b_6$. These symmetries determine the possiblilities for all $J_i$.

	\begin{obs}
		If $J_i = \{a_i,a_{i+1},t,\alpha\}$ then $\alpha$ must be one of the two vertices in $\{a_0,a_2,a_4,a_6\}$ opposite the odd vertex among $a_i,a_{i+1}$ (here we are using the notation $a_i = a_{i-8}$ for $8\leq i < 16$).
	\end{obs}

	\begin{cor}
		\label{triados}
		For each $0\leq i < 8$ at least one of the triangles $\{a_1,a_{i+1},\alpha_0\}$ and $\{a_1,a_{i+1},\alpha_1\}$ appears in the boundary of $t$ (where $\alpha_0$ and $\alpha_1$ are the two vertices in $\{a_0,a_2,a_4,a_6\}$ opposite the odd vertex among $a_i,a_{i+1}$.
	\end{cor}

	\noindent Now notice that the triangles $\{a_i,a_{i+1},b\}$ are all contained in the boundary of $t$ and span a surface homeomorphic to a disk whose boundary is the cycle $a_0,a_1,\dots,a_7$. Therefore the remaining triangles in the boundary of $t$ must also span a disk with the same boundary.\\

	Applying \ref{triados} with $i=0$ we get two cases,

	Case $1$: Triangle $\{a_0,a_1,a_4\}$ appears in the boundary of $t$. Using \ref{triados} with $i=2$ we have at least one of $\{a_2,a_3,a_0\}$ or $\{a_2,a_3,a_6\}$ must appear. Notice that the edge $a_1a_4\in\{a_0,a_1,a_4\}$ crosses both $a_3a_0\in\{a_2,a_3,a_0\}$ and $a_3a_6\in\{a_2,a_3,a_6\}$. This is a contradiction, as the triangles would not form a disk with boundary $a_0,a_1,\dots,a_7$.\\

	Case $2$: Triangle $\{a_0,a_1,a_6\}$ appears in the boundary of $t$. Using \ref{triados} with $i=7$ we have at least one of $\{a_7,a_0,a_2\}$ or $\{a_7,a_0,a_4\}$ must appear. Notice that the edge $a_1a_6\in\{a_0,a_1,a_6\}$ crosses both $a_7a_2\in\{a_7,a_0,a_2\}$ and $a_7a_4\in\{a_7,a_0,a_4\}$. This is a contradiction, as the triangles would not form a disk with boundary $a_0,a_1,\dots,a_7$.\\

\end{proof}

We finish this section by proving the conditions of \ref{segundo} are met, thus showing the conditions for \ref{primero} are met and thus proving \ref{main}.

\begin{lem}
	Let $(V,\Delta)$ and $(V',\Delta')$ be triangulations of $3$-manifolds and $f:\Delta_3\rightarrow \Delta'_3$ an intersection preserving map. Then $F$ induces bijections between $\Delta_i$ and $\Delta'_i$ for $i\geq1$.
\end{lem}

\begin{proof}
	For $i=3$ the claim is clear because $F$ coincides with $f$. 

	For $i=2$ notice every $2$-simplex is the intersection of two facets. Because $f$ is intersection preserving we have that $f$ induces a bijection between the pairs of facets whose intersection is a $1$-simplex.

	For $i=3$ notice every $1$-simplex is the intersection of all facets containing it, these facets span a complex isomorphic to $CW_k$, on the other hand every complex isomorphic to $CW_k$ in $(V,\Delta)$ corresponds to the simplex containing the two vertices that belong to every facet. Because of \ref{conditions} we have that $f$ induces a bijection between complices isomorphic to $CW_k$, therefore $F$ induces a bijection between $\Delta_1$ and $\Delta'_1$.

\end{proof}